\documentclass[conference]{IEEEtran}
\IEEEoverridecommandlockouts
\usepackage{cite}
\usepackage{amsmath,amssymb,amsfonts}
\usepackage{algorithm}
\usepackage{algpseudocode}
\usepackage{amsthm}
\usepackage{graphicx}
\usepackage{enumitem}
\usepackage{hyperref}
\usepackage{textcomp}
\usepackage{subcaption}
\usepackage{caption}

\usepackage{xcolor}
\def\BibTeX{{\rm B\kern-.05em{\sc i\kern-.025em b}\kern-.08em
    T\kern-.1667em\lower.7ex\hbox{E}\kern-.125emX}}
\begin{document}

\title{Communication efficient quasi-Newton distributed optimization based on the Douglas-Rachford envelope}

\author{Dingran Yi and Nikolaos M. Freris\thanks{School of Computer Science, University of Science and Technology of China, Hefei, Anhui, 230027, China. Emails: \texttt{ ydr0826@mail.ustc.edu.cn, nfr@ustc.edu.cn}.}}
%\and

\maketitle

\begin{abstract}
We consider distributed optimization in the client-server setting. By use of Douglas-Rachford splitting to the dual of the sum problem, we design a BFGS method that requires minimal communication (sending/receiving one vector per round for each client). Our method is line search free and achieves superlinear convergence. Experiments are also used to demonstrate the merits in decreasing communication and computation costs.
\end{abstract}


\section{Introduction}
Distributed optimization has received a lot of attention in signal processing, machine learning, and control. A common setting is a set of users with local data and computational capabilities coordinated by a server to
\begin{align}
    \underset{x\in\mathbb{R}^d}{\text{minimize}}\,\,\sum_{i=1}^m f_i(x)+\frac{\lambda}{2}\left\|x\right\|^2\label{pro1}
\end{align}where $f_i$ pertains to client~$i$ and $\frac{\lambda}{2}\left\|\cdot\right\|$ is a regularizer to reduce model complexity.\\There is a wealth of first order methods (e.g., gradient descent-based) \cite{jakovetic2014fast} developed for (1), due to their efficiency and simplicity. Nonetheless, they suffer from slow convergence. A natural choice to remedy this (that we also adopt in this paper) is to resort to second order, i.e., Newton’s method, which can achieve superlinear rate. Nonetheless, this poses serious challenges in terms of computation and communication costs. For second order methods, the server needs to compute the Newton direction. The issue is that this requires to communicate Hessian matrices in addition to gradients \cite{zhang2022distributed} ($\mathcal{O}(d^2)$ cost), which is unsuitable for high-dimensional problems. Several methods have been devised to alleviate this. In \cite{liu2023communication}, the server updates the aggregate Hessian over $d$ rounds during which clients communicate a single column of their local Hessian. However, the incurred delay is undesirable when $d>>1$. There are also methods that apply compression \cite{safaryan2022fednl,islamov2023distributed} and SVD decomposition \cite{agafonov2022flecs, dal2024shed} on local Hessians before transmitting to the server but they cannot guarantee that the cost can be reduced to $\mathcal{O}(d)$. Also SVD decomposition imposes additional computational overhead on the user devices. Another approach is where clients compute a direction locally and the server aggregates, with some criterion to ensure convergence: \cite{crane2020dino} considers the angle between local direction and aggregate gradient, while \cite{eisen2019primal} requires a series of additional communication exchanges per round (i.e., an inner loop). Nonetheless, only linear convergence can be established. In conclusion, the aforementioned methods require computing Hessian matrices which imposes substantial computational burden on user devices. Quasi-Newton (that approximates the curvature from gradients) can ease this problem to some extent \cite{eisen2017decentralized,soori2020dave}. However, these methods all consider local problems as individuals and then try to estimate the sum. It reflects
on the slow convergence rate or heavy communication cost. In this paper, we adopt a different approach. We consider the Douglas-Rachford (D-R) splitting \cite{eckstein1992douglas} on the dual problem of (1) and use the corresponding envelope function \cite{patrinos2014douglas} as the target of quasi-Newton optimization (based on the BFGS algorithm \cite[Chapter~6]{wright2006numerical}). This promotes a decoupling that is key to a communication-efficient distributed implementation. In specific, clients locally compute and communicate gradients pertaining to the envelope and the server performs the update. Under standard assumptions, we establish global convergence which superlinear rate while maintaining the communication cost at minimum (single upload/download of a vector of size $d$ per client for each round). Another novel contribution is that this is achieved without line search, the significance of which is elaborated next.\\Backtracking line search is indispensable to ensure global convergence (i.e., from an arbitrary starting point) with (asymptotic) quadratic/superlinear rate in Newton/quasi-Newton methods \cite{wright2006numerical}. In the distributed setting, this results in an `inner-loop’ of additional communication exchanges per round \cite{crane2020dino}: this not only incurs extra communication and computation burden, but also slows down the algorithm in terms of actual time. In this regard, \cite{liu2023communication,soori2020dave} establish convergence when initiating near optimality (so that line search is not needed and a unit step size can be used); this is quite a restrictive assumption in practical scenarios. \cite{du2024distributed} proposes a line search free BFGS method based on Hessian-vector products and greedy selection from base vectors. Nonetheless, in the distributed setting this would require to communicate approximated Hessian matrices. Besides, \cite{polyak2020new, zhang2022distributed} adopt an adaptive stepsize selection (in the place of line search) based on gradient norm. However, this is only applicable for (exact) Newton’s method (thus requires matrix exchanges; see also Sec. V). Our approach is more similar to the MBFGS in \cite{zhang2005globally,liu2010convergence} which checks the function value only once to determine the stepsize (thus the communication overhead is just one scalar). We go a step beyond this, by designing a new criterion that decreases the computation cost (see Sec. II.B and Fig. 2).\\
\textbf{Contributions}:
\begin{enumerate}
    \item We design a second order method for distributed optimization via applying BFGS on the Douglas-Rachford envelope of the dual problem. We demonstrate that it attains an efficient implementation with minimal communication costs per round.
    \item We establish global convergence with superlinear rate \emph{without line search}. This is key to communication efficiency and is attained by a new adaptive stepsize selection mechanism featuring low computational effort.
    \item Experiments demonstrate noticeable advantages in terms of communication and computation savings over leading baseline methods.
    \end{enumerate}
    \section{Algorithm Establishment}
    Problem (\ref{pro1}) can be  reformulated as
    \begin{align}
        \underset{x\in\mathbb{R}^{md},\theta\in\mathbb{R}^d}{\text{minimize}}&\,\,\,\,F(x)+\frac{\lambda}{2}\left\|\theta\right\|^2\notag\\
        \text{s.t.}&\,\,\,\,x_i-\theta=0,\,\,\, i=1,\hdots,m\label{prob2}
    \end{align}
    where $F(x)=\sum_{i=1}^m f_i(x_i)$ and $x=[x_1^\top,\hdots,x_m^\top]^\top$, We also define $\hat{x}:=\frac{1}{m}\sum_{j=1}^m x_j\in\mathbb{R}^{d}$ and $\bar{x}:=(\hat{x}, \hdots, \hat{x})\in\mathbb{R}^{md}$, with the same notation applying to averaging other user variables (at the server). In the following, we first show that solving the dual problem by means of quasi-Newton minimization of the Douglas-Rachford envelope admits an efficient distributed implementation, and proceed to discuss the rule for determining the stepsize so as to avoid additional communication/computation costs (pertaining to line search).
    \subsection{Distributed method based via BFGS on the dual problem}
    The dual of (2) is equivalent to the following problem: 
    \begin{align}
       \underset{y\in\mathbb{R}^{md}}{\text{minimize}}\,\,\,\,h_1(y)+h_2(y), \label{prob3}
    \end{align}where $h_1(y):=F^\star(-y), h_2(y):=\frac{1}{2\lambda}\left\|\sum_{i=1}^m y_i\right\|^2$ comes from the dual of quadratic (all norms are Euclidean in this paper). It follows from \cite{patrinos2014douglas} that this is equivalent to finding a stationary point of the \emph{Douglas-Rachford envelope}, solving  (\ref{prob3}) is equivalent to finding the stationary point of the following Douglas-Rachford envelope function
    \begin{align}
        H_\gamma^{\text{DR}}(y):=h_2^\gamma(y)-\gamma\left\|\nabla h_2^\gamma(y)\right\|^2+h_1^\gamma(y-2\gamma\nabla h_2^\gamma(y)),\label{DRE}
    \end{align}where $h^\gamma(y):=\underset{z}{\text{inf}}\left\{h(z)+\frac{1}{2\gamma}\left\|y-z\right\|^2\right\}$. Note that using gradient descent to minimize (\ref{DRE}) gives rise to the celebrated ADMM method \cite{patrinos2014douglas}. Thus, ADMM is a first order method that can only attain linear convergence (this can also be understood by the fact that the dual update is carried via gradient ascent). In this paper, we consider second order acceleration by means of applying a quasi-Newton method for minimizing the envelope function based on gradient evaluation. The latter can be calculated as:
    \begin{align*}
        \nabla H_\gamma^{\text{DR}}(y)=\frac{1}{\gamma}\left(I-2\tau Q\right)\left(y-\tau\bar{y}-\text{prox}_{\gamma h_1}(y-2\tau\bar{y})\right),
    \end{align*}
     where $Q:=\frac{1}{m}\textbf{1}_m\textbf{1}_m^\top\otimes I_d\in\mathbb{R}^{md\times md}$ is the averaging matrix and $\tau:=\frac{m\gamma}{m\gamma+\lambda}$. Distributed implementation is possible in view of the fact that $\text{prox}_{\gamma h_1}(y-2\tau\bar{y})$ can be computed at the server based on local computations carried (in parallel) at the users, For
     \begin{align}
         x_i=\underset{x_i}{\text{argmin}}\left\{f_i(x_i)+\left(y_i-2\tau\hat{y}\right)^\top x_i+\frac{\gamma}{2}\left\|x_i\right\|^2\right\},\label{localcomp}
     \end{align}which are intrepreted as \emph{primal variables} (i.e., the local models in (\ref{prob2})), $\text{prox}_{\gamma h_1}(y-2\tau\bar{y})=y-2\tau\bar{y}+\gamma x$. To conclude:$$\nabla H_\gamma^{\text{DR}}(y)=\left(\frac{\tau}{\gamma}-\frac{2\tau^2}{\gamma}\right)\bar{y}-x+2\tau\bar{x}.$$Moreover, we change the order of the three steps of BFGS update for a consistent description of our method:
     \begin{subequations}
     \begin{align}
     B_{k}^{-1}&=B_{k-1}^{-1}+\frac{(s_{k-1}^\top z_{k-1}+z_{k-1}^\top B_{k-1}^{-1}z_{k-1})(s_{k-1}s_{k-1}^\top)}{(s_{k-1}^\top z_{k-1})^2}\notag\\&\,\,\,\,\,\,\,-\frac{B_{k-1}^{-1}z_{k-1}s_{k-1}^\top+s_{k-1}z_{k-1}^\top B_{k-1}^{-1}}{s_{k-1}^\top z_{k-1}},\label{step3}\\
         p^k&=B_k^{-1}\nabla H_\gamma^{\text{DR}}(y^k),\label{step1}\\
         y^{k+1}&=y^k-\eta^kp^k,\label{step2}
     \end{align}
     \end{subequations}
     with $s_k=y^{k+1}-y^k, z_k=\nabla H_\gamma^{\text{DR}}(y^{k+1})-\nabla H_\gamma^{\text{DR}}(y^{k})$. Here $B_k^{-1}$ is to approximate the inverse of Hessian and $p^k$ represents the update direction. Note that to initialize the iteration, we need two different variables $y^0$, $y^1$ ($y^1$ doesn't need to be determined from $y^0$) and clients run (\ref{localcomp}) so that server can get $\nabla H_\gamma^{\text{DR}}(y^0)$ and $\nabla H_\gamma^{\text{DR}}(y^1)$; also client~$i$ needs to store $y_i^1-2\tau\hat{y}_i^1$ for later update.\subsection{Stepsize selection}
In BFGS, superlinear convergence is asymptotic (for some neighborhood of the optimal solution, where a unit stepsize can be used). For global convergence, backtracking line search based on evaluating the function value is rudimentary. In our setting, this is highly unattractive as it would require to run (\ref{localcomp}) multiple times, thus inducing delay and additional communication and computation costs. To remedy this, we devise a new adaptive stepsize selection mechanism with minimal costs and represent it in Alg.~1. Given a constant $\sigma\in\left(0,\frac{1}{2}\right)$, at the beginning of round~$k$, the server calculates~(line~2):
\begin{align}
  q_{k-1}=&\frac{\left\|s_{k-1}-B_{k-1}^{-1}z_{k-1}\right\|}{\left\|B_{k-1}^{-1}s_{k-1}\right\|}+\frac{1}{\gamma}\left\|\eta^{k-1}p^{k-1}\right\|\notag\\&+\left\|\nabla H_\gamma^{\text{DR}}(y^{k-1})\right\|,\label{var_c}  
\end{align}
In round~$k$ after step (\ref{step1}), server decides stepsize with the two conditions:
\begin{align*}
    q_{k-1}&\geq \frac{(1-2\sigma)(p^k)^\top\nabla H_\gamma^{\text{DR}}(y^k)}{4\left\|p^k\right\|^2},\tag{$\mathcal{A}$}\label{condition1}\\
    H_\gamma^{\text{DR}}(y^k-p^k)&\leq H_\gamma^{\text{DR}}(y^k)-\sigma (p^k)^\top\nabla H_\gamma^{\text{DR}}(y^k).\tag{$\mathcal{B}$}\label{condition2}
\end{align*}Based on this rule we can establish that
\begin{align*}
    H_\gamma^{\text{DR}}(y^k)=&\frac{m\lambda^2-m^2\gamma^2}{2\gamma(m\gamma+\lambda)^2}\left\|\bar{y}^k\right\|^2+\frac{\gamma}{2}\left\|x^k\right\|^2-F(x^k)\\&-(x^k)^\top(y^k-\frac{2m\gamma}{m\gamma+\lambda}\bar{y}^k+\gamma x^k).
\end{align*}And we define $v_i^k=-\frac{\gamma}{2}\left\|x_i^k\right\|^2-f_i(x_i^k)-(x_i^k)^\top(y_i^k-2\tau\hat{y}^k)$ to represent client~$i$'s part of the formula above. Our mechanism can be summarized to be: 
if \ref{condition1} $\lor ^\neg$\ref{condition2}, let the stepsize $\eta^k$ be $\frac{\delta (p^k)^\top \nabla H_\gamma^{\text{DR}}(y^k)}{\left\|p^k\right\|^2}$. Otherwise, the server takes unit stepsize. To be more specific, the server first check whether \ref{condition1} holds, if it does, we have determined that $y^{k+1}=y^k-\frac{\delta (p^k)^\top \nabla H_\gamma^{\text{DR}}(y^k)}{\left\|p^k\right\|^2}p^k$. To obtain $\nabla H_\gamma^{\text{DR}}(y^{k+1})$, instead of sending $y_i^{k+1}-2\tau\hat{y}^{k+1}$ to client to implement (\ref{localcomp}), we only use the difference, i.e., $\Delta_i^k:=\eta^k(p_i^k-2\tau\hat{p}^k)$ (line~6). If \ref{condition1} doesn't hold, server further checks whether \ref{condition2} holds, for which we first need try whether unit stepsize can be taken (line~9). In this case, client only sends back $v_i^{k+1}$ first (line~10) to save communication cost. If \ref{condition2} holds, client doesn't need extra computation. If not, client implements (\ref{localcomp}) one more time to obtain $x_i^{k+1}$ (line~16). From this we can also see sending $\Delta_i^k$ (line~9) is communication saving, because if \ref{condition2} doesn't hold, server just needs to send one more scalar (line~15).\\
\begin{algorithm}[t]
	\caption{QND2R (server view)} 
 \textbf{initialization}: $y^0, y^1,\nabla H_\gamma^{\text{DR}}(y^0),\nabla H_\gamma^{\text{DR}}(y^0), B_0$
	\begin{algorithmic}[1]
		\For {$k=1,2,3,\hdots$}
  \State{calculate $q_{k-1}$ based on (\ref{var_c})}
\State{update $B_k^{-1}$ and $p^k$ based on (\ref{step3}) and (\ref{step1})}
  \State{set $\eta^{k}=\frac{\delta \left(p^{k}\right)^\top \nabla H_\gamma^{\text{DR}}(y^k)}{\left\|p^{k}\right\|^2}$}
  \If {\ref{condition1}}
  \State{\underline{send} \ref{condition1} and $\Delta_i^k=\eta^k\left(p_i^{k}-2\tau\hat{p}_i^{k}\right)$ to run Alg.~2}
  \State{\underline{receive} $\left\{x_i^{k+1}, v_i^{k+1}\right\}$, go to line~18}
  \Else
  \State{\underline{send} \ref{condition1} and $\Delta_i^k=p_i^{k}-2\tau\hat{p}_i^{k}$ to run Alg.~2}
  \State{\underline{receive} $v_i^{k+1}$}
  \EndIf
  \If {\ref{condition2}}
  \State{\underline{send} \ref{condition2} and \underline{receive} $x_i$, $\eta^k=1$}
  \Else
  \State{\underline{send} \ref{condition2} and $\eta^k$}
  \State{discard $v_i^{k+1}$ in line~10 and \underline{receive} $\left\{x_i^{k+1},v_i^{k+1}\right\}$}
  \EndIf
  %\EndIf
 \State{$y^{k+1}=y^{k}-\eta^{k}p^{k}$}
 \State{$H_\gamma^{\text{DR}}(y^{k+1})=\frac{m\lambda^2-m^2\gamma^2}{2\gamma(m\gamma+\lambda)^2}\left\|\bar{y}^{k+1}\right\|^2+\sum_{i=1}^m v_i^{k+1}$}
 \State{$\nabla H_\gamma^{\text{DR}}(y^{k+1})=\frac{m\lambda-m^2\gamma}{(m\gamma+\lambda)^2}\bar{y}^{k+1}-x^{k+1}+\frac{2m\gamma}{m\gamma+\lambda}\bar{x}^{k+1}$}
  
  %\Statex
  %\texttt{actually here}
  \EndFor

	\end{algorithmic} 
\end{algorithm}
We design in this way because if we only have condition \ref{condition2} and find it doesn't hold, (\ref{localcomp}) needs to be implemented one more time and thus increase computation cost. We tend to establish another condition, which should be easy to check and can indicate whether \ref{condition2} holds to some extend. It can be understood that BFGS enjoys superliear since as the iteration goes, the variable is close to minima and the approximated matrix is close to the Hessian. Recall the expression of $q$ in (\ref{var_c}), the first part measures how close the approximated matrix is to Hessian and the second part represent the distance between variable and minima. However, when we are determining $\eta^k$, $q_k$ is not known yet. We use $q_{k-1}$ to estimate and when it's small enough (i.e., $^\neg$\ref{condition1}), it is very likely that \ref{condition2} holds. Therefore, only $^\neg$\ref{condition1}, we check (\ref{condition2}) and thus save the computation cost. We summarize the steps in Algorithm QND2R (\textbf{Q}uasi \textbf{N}ewton \textbf{D}istributed \textbf{D}ouglas \textbf{R}achford). 
    
\begin{algorithm}[t]
	\caption{(client~$i$)} 
 \textbf{initialization}: $u_i:=y_i^1-\frac{2m\gamma}{m\gamma+\lambda}\bar{y}_i^1$
	\begin{algorithmic}[1]
		
  \State{\underline{receive} input \ref{condition1} and $\Delta_i^k$ from server}
  \State{compute $x_i^{k+1}=\underset{x_i}{\text{argmin}}\left\{f_i(x_i)+(u_i-\Delta_i^k)^\top x_i+\frac{\gamma}{2}\left\|x_i\right\|^2\right\}$}
  \State{$v_i^{k+1}=-\frac{\gamma}{2}\left\|x_i^k\right\|^2-f_i(x_i^k)-(x_i^k)^\top(u_i-\Delta_i^k)$}
  \If{\ref{condition1}}
  \State{\underline{send} $\left\{x_i^{k+1}, v_i^{k+1}\right\}$, let $u_i=u_i-\Delta_i^k$, go to line~16}
  
  \Else
  \State{\underline{send} $v_i^{k+1}$ and \underline{receive} \ref{condition2}}
  \If{\ref{condition2}}
  \State{\underline{send} $x_i$, let $u_i=u_i-\Delta_i^k$}
  \Else
  \State{\underline{receive} $\eta^k$ and compute}
  \Statex{\qquad\,\,\,\,\,$x_i^{k+1}=\underset{x_i}{\text{argmin}}\left\{f_i(x_i)+(u_i-\eta^k\Delta_i^k)^\top x_i+\frac{\gamma}{2}\left\|x_i\right\|^2\right\}$}
  \State{$v_i^{k+1}=-\frac{\gamma}{2}\left\|x_i^{k+1}\right\|^2-f_i(x_i^{k+1})-x_i^\top(u_i-\eta^k\Delta_i^k)$}
  \State{\underline{send} $\left\{x_i^{k+1}, v_i^{k+1}\right\}$, let $u_i=u_i-\eta^k\Delta_i^k$}
  \EndIf
  \EndIf
  \State{exit}

  %\Statex
  %\texttt{actually here}

	\end{algorithmic} 
\end{algorithm}
\section{convergence analysis}
    Our analysis is based on the following.
    \newtheorem{theorem}{Theorem}
    \newtheorem{lemma}{Lemma}
    \newtheorem{property}{Property}
    \newtheorem{assumption}{Assumption}
    \newtheorem{remark}{Remark}
    
    \begin{assumption}\label{ass1}
        Each $f_i$ is twice differentiable, strongly convex, with Lipschitz continuous gradient, i.e., there exist $0\le L_1\le L_2$ s.t. for $i=1,\hdots,m,$ $$L_1I_d\preceq \nabla^2f_i(x)\preceq L_2I_d,\,\,\,\, x\in\mathbb{R}^d.$$
    \end{assumption}
    \begin{assumption}\label{ass2}
        The Hessian of $f_i$ is Lipschitz continuous, i.e., there exists constant $L_3$, s.t. for $i=1,\hdots,m,$$$\left\|\nabla^2f_i(x)-\nabla^2f_i(y)\right\|\leq L_3\left\|x-y\right\|,\,\,\,\,x,y\in\mathbb{R}^{d}.$$
    \end{assumption}
    \begin{property}
        Under Assumption~1, for given $y$ and $x$ obtained from (\ref{localcomp}), we have $\left\|x-\bar{x}\right\|^2+\left\|\sum_{i=1}^m\left(\nabla f_i(x_i)+\frac{\lambda}{m}x_i\right)\right\|^2\leq\left(\frac{\gamma^2}{\tau^2}+\gamma^2+1\right)\left\|\nabla H_\gamma^{\text{DR}}(y)\right\|^2$. Further, suppose that $y^\star$ is the stationary point of (\ref{DRE}), the corresponding $x^\star$ from satisfies $x_1^\star=\cdots=x_m^\star$ and each $x_i^\star$ solves (\ref{pro1}). 
    \end{property}
    \begin{proof}
    Since $\nabla H_\gamma^{\text{DR}}(y)=(\tau/\gamma-2\tau^2/\gamma)\bar{y}-x+2\tau\bar{x}$ and $\bar{y}=\textbf{1}_m\otimes\hat{y}$, $\bar{x}=\textbf{1}_m\otimes\hat{x}$, we have $\left\|x-\bar{x}\right\|^2\leq\left\|\nabla H_\gamma^{\text{DR}}(y)\right\|$. From (\ref{localcomp}) we obtain $\nabla f_i(x_i)+y_i-2\tau\hat{y}+\gamma x_i=0$, therefore, adding from $1$ to $m$ over $i$ and substitute $\tau$ with $\frac{m\gamma}{m\gamma+\lambda}$ we have
    \begin{align*}
        0=&\sum_{i=1}^m\left(\nabla f_i(x_i)+y_i-2\tau\hat{y}+\gamma x_i\right)\\=&m\hat{y}-2m\tau\hat{y}+m\gamma\hat{x}+\sum_{i=1}^m\nabla f_i(x_i)\\=&\frac{m\gamma+\lambda}{m}\sum_{i=1}^m\left(\left[\nabla H_\gamma^{\text{DR}}(y)\right]_i+x_i\right)-m\gamma\hat{x}+\sum_{i=1}^m\nabla f_i(x_i),
    \end{align*}where $\left[\nabla H_\gamma^{\text{DR}}(y)\right]_i$ means a sub-vector from entry $id-d+1$ to $id$. Therefore,
    \begin{align*}
     &\left\|\sum_{i=1}^m\left(\nabla f_i(x_i)+\frac{\lambda}{m}x_i\right)\right\|^2\\\leq&\frac{(m\gamma+\lambda)^2}{m^2}\left\|\nabla H_\gamma^{\text{DR}}(y)\right\|^2+\gamma^2\left\|x-\bar{x}\right\|^2\\\leq&\left(\frac{\gamma^2}{\tau^2}+\gamma^2\right) \left\|\nabla H_\gamma^{\text{DR}}(y)\right\|^2.  
    \end{align*}
     If $\nabla H_\gamma^{\text{DR}}(y)=0$, we have $x=\bar{x}$, i.e., $x_1=\cdots=x_m$. Meanwhile we have $\sum_{i=1}^m\left(\nabla f_i(x_i)+\frac{\lambda}{m}x_i\right)=0$ so each $x_i$ solves (\ref{pro1}). 
    \end{proof}
    \begin{property}
    Under Assumptions 1 and 2, $f_i^\star$ is strongly convex with Lipschitz continuous gradient with parameters $\frac{1}{L_2}$ and $\frac{1}{L_1}$. Meanwhile, $\nabla^2f^\star_i$ exists and is continuous with parameter~$\frac{L_3}{L_1^3}$
    \end{property}
    \begin{proof}
        The first part can be directly obtained from \cite{zhou2018fenchel}. $f_i$ is closed and  strongly convex, for any $y$, there exists a unique $x$ such that $y=f_i(x)$, therefore, $\nabla f^\star_i(y)=x$, which means the gradient of $f_i^\star$ exists. For $x_0$ and any sequence $\left\{x_j\right\}$ that converge to $x_0$, we have (assume that $x_j\neq x_0$ for any $j$) $$\frac{\left\|\nabla^2 f_i(x_0)(x_j-x_0)-(\nabla f_i(x)-\nabla f_j(x_0))\right\|}{\left\|x_j-x_0\right\|}\rightarrow 0.$$For $y_0=\nabla f_i(x_0)$, and the sequence $\left\{y_j\right\}$ such that $y_j=\nabla f_i(x_j)$, we have
        $$\frac{\left\|\nabla^2f_i(x_0)\left(\nabla f_i^\star(y_j)-\nabla f_i^\star(y_0)\right)-(y_j-y_0)\right\|}{\left\|x_j-x_0\right\|}\rightarrow 0.$$$\nabla^2f_i(\cdot)$ is upper and lower bounded, so$$\frac{\left\|\left(\nabla f_i^\star(y_j)-\nabla f_i^\star(y_0)\right)-\left(\nabla^2f_i(x_0)\right)^{-1}(y_j-y_0)\right\|}{\left\|y_j-y_0\right\|}\rightarrow 0,$$which means $\nabla^2 f_i^\star(y_0)=\left(\nabla^2 f_i(x_0)\right)^{-1}$.
        For any $x_0$ and $x_1$, since
        $
            \left\|I-(\nabla^2f_i(x_1))^{-1}\nabla^2f_i(x_0)\right\|\leq\frac{L_3}{L_1}\left\|x_1-x_0\right\|
        $, we have $\left\|(\nabla^2f_i(x_0))^{-1}-(\nabla^2f_i(x_1))^{-1}\right\|\leq \frac{L_3}{L_1^2}\left\|x_1-x_0\right\|$. Substituting with conjugate function and $y$ variables, we obtain $\left\|\nabla^2f_i^\star(y_0)-\nabla^2f_i^\star(y_1)\right\|\leq \frac{L_3}{L_1^3}\left\|y_1-y_0\right\|$.
    \end{proof}
    Then we show the properties of the envelope function $H_\gamma^{\text{DR}}$ defined in (\ref{DRE}).
    \begin{lemma}
        Under Assumptions~\ref{ass1} and \ref{ass2}, by choosing $\gamma=\frac{\lambda}{3m}$, $H_\gamma^{\text{DR}}$ is strongly convex with Lipschitz continuous gradient with parameters $\min\left\{\frac{1}{8\gamma},\frac{1}{L_2+\gamma}\right\}$ and $\frac{1}{\gamma}$ respectively. Meanwhile the Hessian is continuous with parameter $\frac{L_3L_2^3}{L_1^3}$.
    \end{lemma}
    \begin{proof}
    Based on Property~2, we can further calculate that 
        \begin{align*}
            &\nabla^2H_\gamma^{\text{DR}}(y)=\\&\gamma^{-1}(I-2\tau Q)\left((I-\tau Q)-\left(I+\gamma\nabla^2h_1(y)\right)^{-1}(I-2\tau Q)\right).
        \end{align*}Since $h_1(\cdot)$ is continuous with parameter $\frac{L_3}{L_1^3}$, with the same trick used in Property~2, we obtain $\nabla^2H_\gamma^{\text{DR}}(y)$ is continuous with parameter $\frac{L_3L_2^3}{L_1^3}$. Next, we calculate $z^\top\nabla^2H_\gamma^{\text{DR}}(y)z$ to show that $\nabla^2H_\gamma^{\text{DR}}(y)$ is uniformly lower and upper bounded. 
        \begin{align*}
            &\gamma z^\top\nabla^2H_\gamma^{\text{DR}}(y)z\\=&z^\top z-\left(3\tau-2\tau^2\right)\bar{z}^\top\bar{z}\\&-(z-2\tau\bar{x})^\top\left(I+\gamma\nabla^2h_1(y)\right)^{-1}(z-2\tau\bar{x}).
        \end{align*}
        Since $\frac{1}{1+\gamma/L_1}I\preceq\left(I+\gamma\nabla^2h_1(y)\right)^{-1}\preceq\frac{1}{1+\gamma/L_2}I$ and by choosing $\gamma=\frac{\lambda}{3m}$, $\tau=\frac{m\gamma}{m\gamma+\lambda}=\frac{1}{4}$, we have
        \begin{align*}
            \gamma z^\top\nabla^2H_\gamma^{\text{DR}}(y)z\leq z^\top z-\frac{5}{8}\bar{z}^\top\bar{z}-\frac{z^\top z-\frac{3}{4}\bar{z}^\top\bar{z}}{1+\gamma/L_1}\leq z^\top z,
        \end{align*}
        which means $\nabla^2H_\gamma^{\text{DR}}(y)$ is upper bounded by $\frac{1}{\gamma}I$. Similarly, we obtain
        \begin{align*}
            \gamma z^\top\nabla^2H_\gamma^{\text{DR}}(y)z\geq& z^\top z-\frac{5}{8}\bar{z}^\top\bar{z}-\frac{z^\top z-\frac{3}{4}\bar{z}^\top\bar{z}}{1+\gamma/L_2}\\\geq&\min\left\{\frac{1}{8},\frac{\gamma}{L_2+\gamma}\right\}\bar{z}^\top\bar{z},
        \end{align*}
        so $\nabla^2H_\gamma^{\text{DR}}(y)$ is lower bounded by $\min\left\{\frac{1}{8\gamma},\frac{1}{L_2+\gamma}\right\}I$.
    \end{proof}
        A positive attribute of this analysis is that the hyperparameter can be easily selected without any loss on properties of the individual loss functions (i.e., using only $\lambda,m$ that are directly accessible by the server).
\begin{theorem}
    Under Assumptions~1 and 2, by choosing $\gamma=\frac{\lambda}{3m}$,$\sigma\in\left(0,\frac{1}{2}\right),\delta\in(0,\frac{\lambda}{3m})$, and $B_0$ to be some positive definite matrix, the sequence $\left\{y^k\right\}$ generated by Alg.~1 converges superlinearly to the unique minima $y^\star$ of (\ref{DRE}). And for large enough $k$, unit stepsize is always taken.
\end{theorem}
\begin{proof}
    It follows from (\ref{step3}) that $B_k$ is positive definite for all $k$. We proceed to establish that $H_\gamma^{\text{DR}}(y^k)$ is a decreasing sequence. If $\eta^k=\frac{\delta(p^k)^\top\nabla H_\gamma^\text{DR}(y^k)}{\left\|p^k\right\|^2}$, we have
    \begin{align*}
        &H_\gamma^{\text{DR}}(y^{k+1})-H_\gamma^{\text{DR}}(y^{k})=\nabla H_\gamma^\text{DR}(\xi^k)^\top s_k\\\leq& \nabla H_\gamma^\text{DR}(y^k)^\top s_k+\left\|\nabla H_\gamma^\text{DR}(\xi^k)-\nabla H_\gamma^\text{DR}(y^k)\right\|\left\|s_k\right\|\\\leq&(\eta^k)^2\left\|p^k\right\|^2-\eta^k\nabla H_\gamma^\text{DR}(y^k)^\top p^k\leq(\delta-1)\eta^k\nabla H_\gamma^\text{DR}(y^k)^\top p^k,
    \end{align*}where $\xi^k$ is on the segment between $y^k$ and $y^{k+1}$. If $\eta^k=1$, of course the sequence is decreasing. We then show the norm of gradient will converge to zero. According to \cite[Thm 2.1]{byrd1989tool}, there exist $\beta_1,\beta_2,\beta_3$ such that $$\left\|B_ks_k\right\|\leq\beta_1\left\|s_k\right\|,\,\,\, \beta_2\left\|s_k\right\|^2\leq s_k^\top B_ks_k\leq\beta_3\left\|s_k\right\|^2$$ hold for infinitely many $k$. Denoting the subsequence $\left\{k'\right\}$, we have$$\eta^{k'}\geq\min\left\{1,\frac{\delta (p^{k'})^\top B_{k'} p^{k'}}{\left\|p^{k'}\right\|}\right\}\geq\min\left\{1,\delta\beta_2\right\}.$$ $-B_{k'}s_{k'}=\eta^{k'}B_{k'}p^{k'}=\eta^{k'}\nabla H_\gamma^{\text{DR}}(y^{k'})$, so $\left\|\nabla H_\gamma^{\text{DR}}(y^{k'})\right\|\leq \left\|p^{k'}\right\|$, which means
    \begin{align*}
    &\nabla H_\gamma^{\text{DR}}(y^{k'})^\top p^{k'}=(p^{k'})^\top B_{k'}p^{k'}=\frac{1}{(\eta^{k'})^2}(s_{k'})^\top B_{k'}s_{k'}\\\geq &\frac{\beta_2}{(\eta^{k'})^2}\left\|s_{k'}\right\|^2=\beta_2\left\|p^{k'}\right\|^2\geq\beta_2\left\|\nabla H_\gamma^{\text{DR}}(y^{k'})\right\|^2.
    \end{align*}In the following, we use the trick of contradiction: if $\left\|\nabla H_\gamma^{\text{DR}}(y^k)\right\|\nrightarrow 0$, since $H_\gamma^{\text{DR}}(y^k)$ is a decreasing sequence, we have $\left\|\nabla H_\gamma^{\text{DR}}(y^{k'})\right\|\nrightarrow 0$. Because
    \begin{align*}
        H_\gamma^{\text{DR}}(y^{k'+1})-H_\gamma^{\text{DR}}(y^{k'})\leq(\delta-1)\eta^{k'}\beta_2\left\|\nabla H_\gamma^{\text{DR}}(y^{k'})\right\|^2,
    \end{align*}we have $H_\gamma^{\text{DR}}(y^{k'})\rightarrow -\infty$, which conflicts with the strong convexity of $H_\gamma^{\text{DR}}$. Therefore $\left\|\nabla H_\gamma^{\text{DR}}(y^k)\right\|\rightarrow 0$.\\We now establish the superlinear convergence, which mainly comes from the fact that a unit stepsize can be always chosen for large enough $k$. Following our algorithm, this is equivalent to showing that for large enough $k$, condition in line~12 holds while that in line~5 does not. From \cite[Thm 3.2]{byrd1989tool}, denoting by $\nabla^2_\star$ the Hessian at $y^\star$, we have $\underset{k\rightarrow \infty}{\lim}\frac{\left\|(B_k-\nabla^2_\star)s_k\right\|}{\left\|s_k\right\|}=0$ and $\left\|B_k\right\|,\left\|B_k^{-1}\right\|$ are uniformly bounded. There exists some $\xi^k$ between $y^k$ and $y^{k+1}$, such that $z_k=\nabla^2H_\gamma^{\text{DR}}(\xi^k)s_k$, then
    \begin{align*}
        \left\|B_ks_k-z_k\right\|\leq\left\|(B_k-\nabla^2_\star)s_k\right\|+\left\|(\nabla^2H_\gamma^{\text{DR}}(\xi^k)-\nabla_\star^2)s_k\right\|,
    \end{align*}and
    \begin{align*}
    \left\|(B_k-\nabla^2_\star)s_k\right\|\leq\left\|B_ks_k-z_k\right\|+\left\|(\nabla^2H_\gamma^{\text{DR}}(\xi^k)-\nabla_\star^2)s_k\right\|.
    \end{align*}By letting $\tilde{q}_k=\frac{\left\|B_ks_k-z_k\right\|}{\left\|s_k\right\|}+M(\left\|\nabla H_\gamma^{\text{DR}}(y^k)\right\|+\frac{1}{\gamma}\left\|\eta^kp^k\right\|)$, where $M:=\frac{L_3L_2^3}{L_1^3}\left(L_2+8\gamma\right)$
    is related to the Hessian continuity parameter $\frac{L_3L_2^3}{L_1^3}$ and strong convexity parameter $\min\left\{\frac{1}{8\gamma},\frac{1}{L_2+\gamma}\right\}$ obtained in Lemma~1, we have
        $\frac{\left\|(B_k-\nabla^2_\star)s_k\right\|}{\left\|s_k\right\|}\leq \tilde{q}_k\rightarrow 0$. Therefore, $q_k:=\frac{\left\|s_k-B_k^{-1}z_k\right\|}{\left\|B_k^{-1}s_k\right\|}+\left\|\nabla H_\gamma^{\text{DR}}(y^k)\right\|+\frac{1}{\gamma}\left\|\eta^kp^k\right\|\rightarrow 0$.
     From the boundedness of $B_k$ and $B_k^{-1}$, $\frac{(p^k)^\top\nabla H_\gamma^{\text{DR}}(y^k)}{\left\|p^k\right\|^2}$ has a uniform lower bound, which means that line~8 in Alg. 1 is eventually executed. Moreover, we have
    \begin{align*}
        \lvert(p^k)^\top\nabla H_\gamma^{\text{DR}}(y^k)-(p^k)^\top\nabla^2_\star p^k\rvert\leq \tilde{q}_k\left\|p^k\right\|^2.
    \end{align*}Since
    \begin{align*}
        &H_\gamma^{\text{DR}}(y^{k}+p^k)-H_\gamma^{\text{DR}}(y^{k})\\=&(p^k)^\top\nabla H_\gamma^{\text{DR}}(y^k)+\frac{1}{2}(p^k)^\top\nabla^2H_\gamma^{\text{DR}}(\xi_k) p^k\\\leq&\frac{1}{2}(p^k)^\top\nabla H_\gamma^{\text{DR}}(y^k)+2\tilde{q}_k\left\|p^k\right\|^2\\=&\sigma(p^k)^\top\nabla H_\gamma^{\text{DR}}(y^k)+2\tilde{q}_k\left\|p^k\right\|^2+\frac{1-2\sigma}{2}(p^k)^\top\nabla H_\gamma^{\text{DR}}(y^k),
    \end{align*}if $\tilde{q}_k\leq -\frac{(2\sigma-1)(p^k)^\top\nabla H_\gamma^{\text{DR}}(y^k)}{4\left\|p^k\right\|^2}$, a unit stepsize is accepted.
\end{proof} 
\section{Experiments}
\begin{figure}[t]
  \centering
  \begin{subfigure}[b]{0.49\linewidth}
    \includegraphics[width=\textwidth]{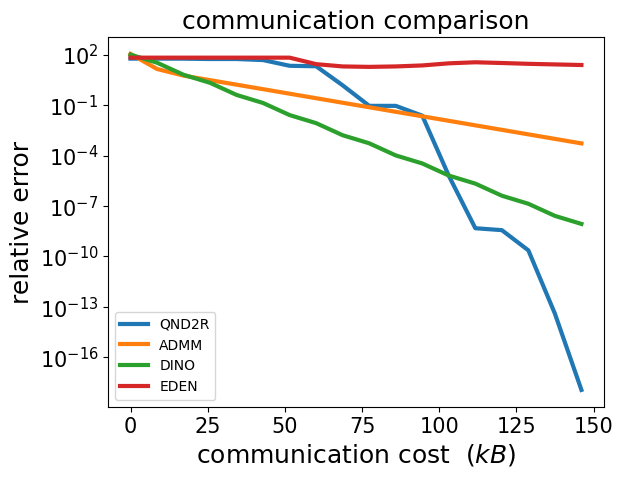}
    \caption{}
  \end{subfigure}
  \begin{subfigure}[b]{0.49\linewidth}
    \includegraphics[width=\textwidth]{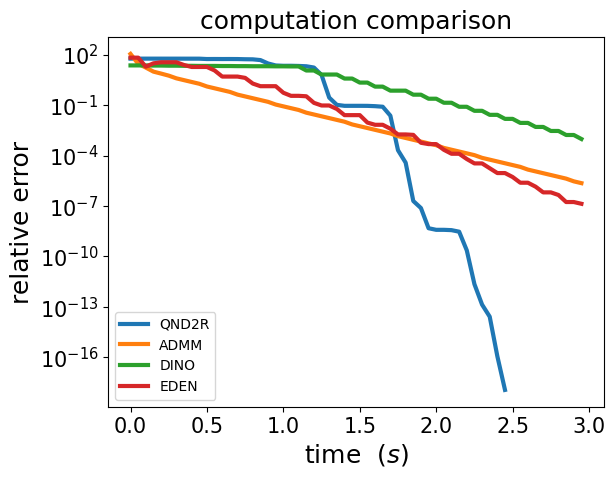}
    \caption{}
  \end{subfigure}\label{vs}
  \caption{Relative error reduction vs. (a) total communication cost and (b) total computation cost (run-time).}
  \label{fig1}
  \begin{subfigure}[b]{0.49\linewidth}
    \includegraphics[width=\textwidth]{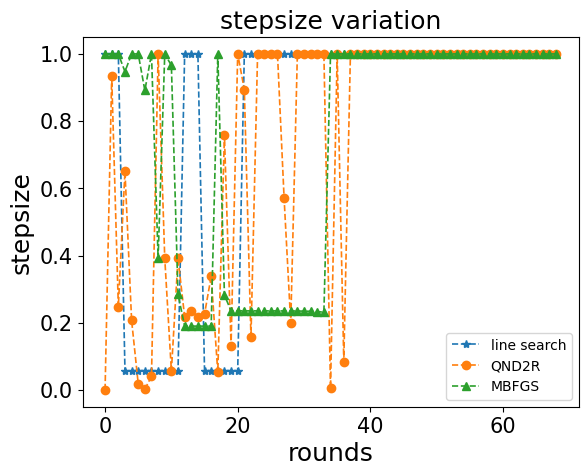}
    \caption{}
  \end{subfigure}
  \begin{subfigure}[b]{0.49\linewidth}
    \includegraphics[width=\textwidth]{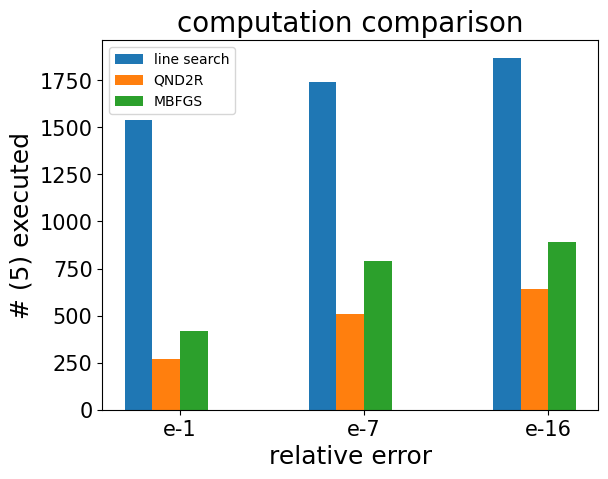}
    \caption{}
  \end{subfigure}
  \caption{Performance of different stepsize mechanism vs. stepsize variation (a) and computation cost (b).}
  \label{fig2}
\end{figure}
        We evaluate our algorithm on a distributed logistic regression problem:$$f_i(x) := \frac{1}{m_i}\sum_{j=1}^{m_i} \left[\ln\left(1+e^{w_j^Tx}\right) +(1-y_j)w_j^Tx\right],$$ where $m_i$ is the number of data points held by each agent and $\{w_j,y_j\}_{j=1}^{m_i}\subset \mathbb{R}^d\times\{0,1\}$ are labeled samples. We used data from \href{https://www.csie.ntu.edu.tw/~cjlin/libsvm/}{LIBSVM}. We take 5,000 data points with dimension $d=22$, and distribute them across $m=10$ agents before ordering by label. We compare with two other papers DINO \cite{crane2020dino}, EDEN \cite{liu2023communication} as well as standard ADMM since they avoid Hessian communication. We don't compare with papers using compression techniques like \cite{safaryan2022fednl,islamov2023distributed,agafonov2022flecs, dal2024shed} because they need compression ratio tuning and matrix decomposition computation. For our method, since the server doesn't contain a model, the relative error is set to be $\left\|\sum_{i=1}^m \nabla f_i(x_i^k)+\frac{\lambda}{m}x_i\right\|^2+\left\|x-\bar{x}\right\|^2$, where the first part represents the gradient and the second part represents the consensus error. For others, the relative error is set to be $\left\|\lambda x^k+\sum_{i=1}^m \nabla f_i(x^k)\right\|^2$. Fig.~1 show the comparison with  algorithms above in terms of communication (a) and computation (b) cost. Our method QND2R gives a clear superlinear convergence and outperforms baselines. Fig.~2 illustrate the effect of our stepsize selection mechanism. We compare with inexact line search and MBFGS used in \cite{zhang2005globally,liu2010convergence}. Fig.~2.a shows the stepsize variation, in which our method finally take a unit stepsize. As expected, QND2R keeps the unit stepsize at the latest stage, however, Fig.~2.b shows our method requires least computation for target accuracy since the computation in our setting can be measured by times (\ref{localcomp}) executed. To be specific, our method gives $35.7\%, 35.4\%, 28.1\%$ computation savings for the three target accuracy respectively compared with MBFGS.

\section{discussions}
\emph{Central model}:\\
The server does not contain a model in our implementation. If one is desirable, any single model or the averaged model, i.e., $\frac{1}{m}\sum_{j=1}^m x_j$ can be considered. The guarantee is because of Property~1: the consensus error and distance to optimality are upper bounded by the gradient norm of the envelope.\\

\emph{Regularizer choice}:\\
In (\ref{pro1}), we consider quadratic instead of general regularizers. This is because when establishing envelope in (\ref{DRE}), we directly use formula of conjugate function to give a concise algorithm as well as clear values for hyperparameters. For another regularizer, one can use the same analysis to obtain the coefficients and hyperparameters for it.\\

\emph{More on stepsize selection}:\\
In our setting, checking the decrease on the function value needs to solve (\ref{localcomp}), which means if unit stepsize does not work in one round, methods that only consider \ref{condition2} will waste the computation. Thus, we establish \ref{condition1} so that we can avoid such waste as much as possible. The effect of this choice is observed in Fig.~2. Although our method takes more conservative stepsize (Fig.~2.a), this is carried with much less computation effort, thus leading to substantial overall savings (Fig.~2.b).

\vspace{12pt}

\bibliographystyle{unsrt}
\bibliography{reference}
\end{document}